\documentclass[11pt,a4paper]{article}
\usepackage[utf8]{inputenc}
\usepackage{amsmath,amsthm,amsfonts,amssymb}
\usepackage{booktabs,hyperref,mathrsfs}

\title{\textbf{Determinants of twisted generalized hybrid weaving knots}}
\author{
	Sahil Joshi\footnote{\textit{Email address:} \texttt{2018maz0001@iitrpr.ac.in}}~~and
	Madeti Prabhakar\footnote{\textit{Email address:} \texttt{prabhakar@iitrpr.ac.in}}}

\newtheorem{thm}{Theorem}
\newtheorem{cor}{Corollary}
\newcommand{\Z}{\mathbb Z}
\newcommand{\s}{\sigma}
\newcommand{\K}{\hat{Q}_3(m_1,-m_2,n,\ell)}
\newcommand{\F}{\mathcal F}

\allowdisplaybreaks

\begin{document}
	
\maketitle

\begin{abstract}
	This article presents a formula for the determinant of the twisted generalized hybrid weaving knot $\K$, which is a closed $3$-braid.
	As a corollary, we prove Conjecture~2 in Singh \& Chbili~\cite[p.\,16]{SC}.\\
	
	\noindent$2020$ \textit{Mathematics Subject Classification}: 57K10
\end{abstract}

\section{Introduction}

Knots and links that are isotopic to closed $3$-braids form an important class.
Recently, Singh \& Chbili~\cite{SC} defined twisted generalized hybrid weaving knots, which form a subset of the set of closed $3$-braids, and studied their HOMFLY-PT polynomials.
In this paper, we find their knot determinants.
We shall review the background needed for the paper.

The Artin $3$-braid group $B_3$ is finitely presented as
\begin{equation*}
	B_3 = \langle \s_1,\s_2 \mid \s_1\s_2\s_1 = \s_2\s_1\s_2 \rangle.
\end{equation*}

We consider the \emph{reduced} Burau representation $\varphi: B_3\to GL_2(\Z[t^{\pm1}])$ under which the braid generators
\begin{equation*}
	\s_1 \mapsto
	\begin{bmatrix}
		-t & 1\\
		0  & 1
	\end{bmatrix}, \qquad\s_2 \mapsto
	\begin{bmatrix}
		1 & 0\\
		t & -t
	\end{bmatrix}.
\end{equation*}

\begin{thm}[see \textsc{Theorem} 3.11 in Birman~\cite{Bir}]\label{Birman}
	If $\beta \in B_3$ and $\Delta_{\mathbf{cl}(\beta)}(t)$ denotes the Alexander polynomial of the knot $\mathbf{cl}(\beta)$ formed by closing the braid $\beta$, then for some integer $k$,
	\begin{equation}\label{eq:AB}
		\pm\,t^k \,(1 + t + t^2) \, \Delta_{\mathbf{cl}(\beta)}(t) = \det(\varphi(\beta) - I),
	\end{equation}
	where $I$ stands for the identity matrix in $GL_2(\Z[t^{\pm1}])$.
\end{thm}

The reader may refer to \cite{Bir} which explains this relationship in detail.
For any knot or link $K$, its determinant $\det(K)$ and Alexander polynomial $\Delta_K(t)$ are well-known invariants of the isotopy class of $K$.
Further, the two are related by $\det(K) = \vert\Delta_K(-1)\vert$.

Let $m_1, m_2$ and $n$ be positive integers and let $\ell$ be any integer.
The \emph{twisted generalized hybrid weaving knot} $\K$ is the knot or link isotopic to the closure of the $3$-strand braid $(\s_1^{m_1}\s_2^{-m_2})^n (\s_1\s_2)^{3\ell}$.

The class $\F = \{\K : m_1, m_2, n\in\Z_+,\,\ell\in\Z\}$ contains special types of knots and links.
For example, $\hat{Q}_3(1,-1,n,0)$ represents the weaving knot $W(3,n)$, and in general, $\hat{Q}_3(m,-m,n,0)$ represents the hybrid weaving knot $\hat{W}_3(m,n)$, $\hat{Q}_3(1,-1,n,\ell)$ is a quasi-alternating knot or link if and only if $\ell \in \{-1,0,1\}$, etc.
More examples are given in Table~\ref{list}.

In Section~2, we apply Theorem~\ref{Birman} to derive knot determinants for the class of twisted generalized hybrid weaving knots.
In Section~3, we prove a conjecture of Singh and Chbili and retrieve determinants of several links and link families.
In the end, we mention an example of twisted generalized hybrid weaving knots that is not quasi-alternating, and vice-versa.

\section{The Main Result}

\begin{thm}\label{det}
	For $\K\in\F$, the knot determinant is given by
	\begin{equation}\label{eq:det}
		\begin{split}
			\det(\K) & = \Bigg\vert
			\left(\frac{2 + m_1m_2 + \sqrt{m_1^2m_2^2+4m_1m_2}}{2}\right)^n +\\
			&\hspace{-2em}
			\left(\frac{2 + m_1m_2 - \sqrt{m_1^2m_2^2+4m_1m_2}}{2}\right)^n + (-1)^{\ell+1}\,2\ \Bigg\vert.
		\end{split}
	\end{equation}
\end{thm}
\begin{proof}
	Let $\varphi$ be the group homomorphism mentioned in Theorem~\ref{Birman} and let
	\begin{equation*}
		A = \varphi(\s_1)\big\vert_{t=-1} = 
		\begin{bmatrix}
			1 & 1\\
			0 & 1
		\end{bmatrix}, \quad
		B = \varphi(\s_2)\big\vert_{t=-1} = 
		\begin{bmatrix}
			1  & 0\\
			-1 & 1
		\end{bmatrix}.
	\end{equation*}
	If $\beta = (\s_1^{m_1}\s_2^{-m_2})^n (\s_1\s_2)^{3\ell}$, then $\mathbf{cl}(\beta) = \K$ and $\varphi(\beta)\big\vert_{t=-1} = (A^{m_1} B^{-m_2})^n (AB)^{3\ell}$.
	Observe that $(AB)^3 = -I$.
	Put
	\begin{align*}
		C = A^{m_1} B^{-m_2} =
		\begin{bmatrix}
			1 + m_1m_2	& m_1\\
			m_2			& 1
		\end{bmatrix}.
	\end{align*}
	The characteristic polynomial of $C$ is $f(x) = x^2 - (2+m_1m_2)x + 1$.
	The eigenvalues of $C$ are
	\begin{equation*}
		\frac{2 + m_1m_2 + \sqrt{m_1^2m_2^2 + 4m_1m_2}}{2},\quad
		\frac{2 + m_1m_2 - \sqrt{m_1^2m_2^2 + 4m_1m_2}}{2}.
	\end{equation*}
	Therefore, $C$ is diagonalizable.
	Hence, there exists an invertible matrix $P$ such that $C = PDP^{-1}$, where $D$ is a diagonal matrix whose diagonal entries are the eigenvalues of $C$.
	Using~\eqref{eq:AB}, we get
	\begin{align*}
		\pm\Delta_{\K}(-1)
		& = \det(\varphi(\beta) - I)\big\vert_{t=-1}\\
		& = \det((A^{m_1} B^{-m_2})^n (AB)^{3\ell} - I)\\
		& = \det((-1)^\ell C^n - I)\\
		& = \det(PD^nP^{-1} - (-1)^\ell I)\\
		& = \det(D^n + (-1)^{\ell+1} I)\\
		& = \left[ \left(\frac{2 + m_1m_2 + \sqrt{m_1^2m_2^2 + 4m_1m_2}}{2}\right)^n +
		(-1)^{\ell+1} \right] \times\\
		& \qquad \left[ \left(\frac{2 + m_1m_2 - \sqrt{m_1^2m_2^2 + 4m_1m_2}}{2}\right)^n + (-1)^{\ell+1} \right]\\
		& = (-1)^{\ell+1} \left(\frac{2 + m_1m_2 + \sqrt{m_1^2m_2^2 + 4m_1m_2}}{2}\right)^n +\\
		& \qquad (-1)^{\ell+1} \left(\frac{2 + m_1m_2 - \sqrt{m_1^2m_2^2 + 4m_1m_2}}{2}\right)^n + 2.
	\end{align*}
	Since $\det(\K) = \vert \Delta_{\K}(-1) \vert$, we obtain~\eqref{eq:det}.
\end{proof}

\section{Applications of Theorem~\ref{det}}

Now we proceed to give some applications of Theorem~\ref{det}.
In this regard, we first recall the Binet formula for $m$-Lucas numbers $(m\in\Z_+)$, which is as follows:

\begin{thm}[Falcon {\cite[Theorem 2.2]{Fal}}]
	For each $m\in\Z_+$, $m$-Lucas numbers $\{L_{m,n} : n=0,1,2,\ldots\}$ are given by the formula $L_{m,n} = \Phi_m^n + (-\Phi_m^{-1})^n$, where $\Phi_m = \frac{m + \sqrt{m^2 + 4}}{2}$.
\end{thm}

If we substitute $m_1=m_2=1$ and $\ell=0$ in Theorem~\ref{det}, then we recover a known formula for the determinant of the weaving knot $W(3,n)$ in terms of the $1$-Lucas numbers, or simply, Lucas numbers.

\begin{cor}[see \cite{KST, JNP}]
	Let $\{L_n : n=0,1,2,\ldots\}$ be the sequence of Lucas numbers.
	Then for the weaving knot $W(3,n)$, we have
	\begin{equation}\label{eq:detW}
		\det(W(3,n)) = L_{2n} - 2.
	\end{equation}
\end{cor}

\begin{proof}
	Note that $\hat{Q}_3(1,-1,n,0) = W(3,n)$.
	Suppose that $m_1=m_2=1$ and $\ell=0$ in Theorem~\ref{det}.
	Then~\eqref{eq:det} reduces to
	\begin{align*}
		\det(\hat{Q}_3(1,-1,n,0))
		& = \left(\frac{3 + \sqrt{5}}{2}\right)^n + \left(\frac{3 - \sqrt{5}}{2}\right)^n - 2\\
		& = \left(\frac{1 + \sqrt{5}}{2}\right)^{2n} + \left(\frac{1 - \sqrt{5}}{2}\right)^{2n} - 2\\
		& = L_{2n} - 2.\qedhere
	\end{align*}
\end{proof}

Singh and Chbili proposed that the knot determinant of $\hat{Q}_3(m,-m,n,\ell)$, which is a twisted hybrid weaving knot, for $\ell\in\{-1,0,1\}$ can be expressed in terms of generalized Lucas numbers, see~\cite[Conjecture\,2]{SC}.
In the next corollary, we give a proof of the same.
\begin{cor}
	Let $\{L_{m,n} : n=0,1,2,\ldots\}$ denote the sequence of $m$-Lucas numbers.
	Then for $\hat{W}_3(m,n), \hat{Q}_3(m,-m,n,\pm1)\in\F$,
	\begin{align}
		\det(\hat{W}_3(m,n)) & = L_{m,2n} - 2, \label{eq:dethW}\\		
		\det(\hat{Q}_3(m,-m,n,\pm1)) & = L_{m,2n} + 2.	\label{eq:detthW}
	\end{align}
\end{cor}

\begin{proof}
	If $m_1=m_2=m$ in Theorem~\ref{det} and $\Phi_m = \frac{m + \sqrt{m^2 + 4}}{2}$, then \eqref{eq:det} gives
	\begin{align*}
		\det(\hat{Q}_3(m,-m,n,\ell)) & = \Bigg\vert
		\left(\frac{2 + m^2 + \sqrt{m^4 + 4m^2}}{2}\right)^n + \\
		& \qquad\quad
		\left(\frac{2 + m^2 - \sqrt{m^4 + 4m^2}}{2}\right)^n + (-1)^{\ell+1}\,2\ \Bigg\vert\\
		& = \big\vert \Phi_m^{2n} + \Phi_m^{-2n} + (-1)^{\ell+1}\,2 \big\vert\\
		& = \big\vert L_{m,2n} + (-1)^{\ell+1}\,2 \big\vert.
	\end{align*}
	Note that $\hat{Q}_3(m,-m,n,0) = \hat{W}_3(m,n)$.
	By substituting $\ell = 0$ and $\ell = \pm1$ in the previous equation, we obtain \eqref{eq:dethW} and \eqref{eq:detthW} respectively.
\end{proof}

Observe that we can also recover the determinant of torus knots and links of type $(2,q)$ from~\eqref{eq:det}.

\begin{cor}
	If $q\in\Z_+$ and $T(2,q)$ is the torus knot or link of type $(2,q)$, then $\det(T(2,q)) = q$.
\end{cor}

\begin{proof}
	Suppose that $m_1=q$, $m_2=1$, $n=1$ and $\ell=0$ in Theorem~\ref{det}.
	Then $\hat{Q}_3(q,-1,1,0) = T(2,q)$ and \eqref{eq:det} yields
	\begin{equation*}
		\det(T(2,q)) = \Bigg\vert
		\frac{2 + q + \sqrt{q^2+4q}}{2} + \frac{2 + q - \sqrt{q^2+4q}}{2} - 2\Bigg\vert = q.\qedhere
	\end{equation*}
\end{proof}

The set $\mathcal{Q}$ of quasi-alternating links is another well-known class of knots and links.
Every alternating knot or link is in $\mathcal{Q}$.
From the classification theorem of quasi-alternating links with braid index at most $3$ (see Baldwin~\cite[\textsc{Theorem} 8.6]{Bal}), it follows that $\hat{Q}_3(1,-m,n,l)$ is quasi-alternating if and only if $\ell\in\{-1,0,1\}$.

Motivated by intellectual curiosity, we give the following result wherein some families of quasi-alternating $3$-braid links are distinguished by their determinants expressed in terms of Lucas numbers.

\begin{cor}
	For any $\hat{Q}_3(1,-5,n,l)\in\F$, we have $\det(\hat{Q}_3(1,-5,n,l)) = L_{4n} + (-1)^{\ell+1}2$.
\end{cor}
\begin{proof}
	After substituting $m_1=1$ and $m_2=5$ in \eqref{eq:det}, we get
	\begin{align*}
		\det(\hat{Q}_3(1,-5,n,\ell)) & = \Bigg\vert
		\left( \frac{7 + \sqrt{45}}{2} \right)^n + \left( \frac{7 - \sqrt{45}}{2} \right)^n + (-1)^{\ell+1} 2\ \Bigg\vert\\
		& = \left(\frac{1 + \sqrt{5}}{2}\right)^{4n} + \left(\frac{1 - \sqrt{5}}{2}\right)^{4n} + (-1)^{\ell+1} 2\\
		& = L_{4n} + (-1)^{\ell+1} 2.\qedhere
	\end{align*}
\end{proof}

In the previous corollary, $\hat{Q}_3(1,-5,1,2)$ represents the well-known Perko's pair $\{10_{161}, 10_{162}\}$.
We conclude this section with Table~\ref{list}.

In Table~\ref{list}, $8_{20}, 10_{125}, 10_{126}$, and $10_{157}$ are quasi-alternating knots, which are nonalternating.
We note that from the classification theorem of quasi-alternating $3$-braid links, \cite[\textsc{Theorem} 8.6]{Bal}, it follows that $10_{139}$ and the Perko's pair $\{10_{161}, 10_{162}\}$ are not quasi-alternating knots.

Moreover, $10_{139}$ is only $3$-colorable and Perko's pair is only $5$-colorable because their determinants are $3$ and $5$, respectively.

It is noteworthy that $10_{139}\in\F$, whilst $10_{139}\notin\mathcal{Q}$.
On the other hand, $\mathbf{cl}(\s_1 \s_2^{-1}\s_1 \s_2^{-3})=6_2\in\mathcal{Q}$ but $6_2\notin\F$.

\begin{table}[h]
	\centering
	\caption{Determinants of some twisted generalized hybrid weaving knots calculated using~\eqref{eq:det}.}\label{list}
	\begin{tabular}{lcr}
		\toprule
		$(m_1,m_2,n,l)$ & $\K$ & $\det(\K)$\\
		\midrule
		$(3,1,1,0)$		& $3_1$		& $3$\\
		$(1,1,2,0)$		& $4_1$		& $5$\\
		$(5,1,1,0)$		& $5_1$		& $5$\\
		$(3,1,1,-1)$	& $5_2$		& $7$\\
		$(7,1,1,0)$		& $7_1$		& $7$\\
		$(3,3,1,1)$		& $7_3$		& $13$\\
		$(3,1,2,0)$		& $8_5$		& $21$\\
		$(1,1,4,0)$		& $8_{18}$	& $45$\\
		$(1,5,1,1)$		& $8_{20}$	& $9$\\
		$(9,1,1,0)$		& $9_1$		& $9$\\
		$(5,3,1,1)$		& $9_3$		& $19$\\
		$(1,1,5,0)$		& $10_{123}$& $121$\\
		$(1,7,1,1)$		& $10_{125}$& $11$\\
		$(5,3,1,-1)$	& $10_{126}$& $19$\\
		$(1,3,1,2)$		& $10_{139}$& $3$\\
		$(1,1,4,-1)$	& $10_{157}$& $49$\\
		$(1,5,1,2)$		& $10_{161} = 10_{162}$& $5$\\
		\bottomrule
	\end{tabular}
\end{table}

\section*{Acknowledgement}
The authors owe special thanks to Dr.\,Vivek K.\,Singh for pointing out that the determinant formula for $3$-strand weaving knots can be generalized to include twisted hybrid weaving knots.

The second author acknowledges the support given by the Science and Engineering Research Board (SERB), Department of Science \& Technology, Government of India under grant-in-aid Mathematical Research Impact Centric Support (MATRICS) with F.\,No.\,MTR/2021/000394.

\pagebreak
\bibliographystyle{plain}
\bibliography{references}

\end{document}